\documentclass{amsart}

\DeclareMathOperator{\pstre}{p-stre}
\DeclareMathOperator{\ev}{ev}
\DeclareMathOperator{\Aut}{Aut}
\newcommand{\Ad}[1]{\text{Ad}#1}
\newcommand{\ddt}{\frac{d}{dt}}

\newcommand{\g}{\mathfrak{g}}
\newcommand{\KK}{\mathcal{K}}
\newcommand{\GG}{\mathcal{G}}
\newcommand{\VV}{\mathcal{V}}
\newcommand{ \TT}{\mathcal{T}}
\newcommand{\OO}{\mathcal{O}}
\newcommand{\R}{\mathbb{R}}
\newcommand{\X}{\mathfrak{X}}
\newcommand{\wt}{\widetilde} 

\DeclareMathOperator{\Gl}{Gl}
\DeclareMathOperator{\rank}{rank}

\newtheorem{theorem}{Theorem}[section] 
\newtheorem{lemma}[theorem]{Lemma} 
\newtheorem{cor}[theorem]{Corollary} 
\newtheorem{prop}[theorem]{Proposition}

\theoremstyle{definition}
\newtheorem{defn}[theorem]{Definition}

\newtheorem{rmk}[theorem]{Remark}

\title[Connection preserving actions]{Connection preserving actions
are topologically engaging} 
	
\author{A. Candel}
\thanks{Research supported by N.S.F. Grant No. DMS-9973086 and
DMS-0049077.}
\address{Department of Mathematics \\ 
	 California State University \\ Northridge, CA 91330, USA}
\email{alberto.candel@csun.edu}
\author{ R. Quiroga--Barranco}  
\address{Departamento de Matem\'aticas \\  
	 CINVESTAV-IPN \\  
	 Apartado Postal 14-740 \\  
	 M\'exico DF 07300 \\  
	 M\'exico}  
\email{quiroga@math.cinvestav.mx}  
\thanks{Research supported by CONACYT Grant No. 32197-E}

\subjclass{Primary: 53C10; Secondary: 53C24}

\keywords{semisimple, group, action, connection, topological engagement,
geometric engagement, rigidity}

\begin{document}

\begin{abstract}
Topologically and geometrically engaging actions have proved to be
useful to obtain rigidity results for semisimple Lie group actions
(see \cite{SZ91}, \cite{CQ99}). We show that the action of a simple
noncompact Lie group on a compact manifold preserving a unimodular
rigid geometric structure of algebraic type (e.g.~a connection
together with a volume density) is topologically engaging on an open
conull dense set.
\end{abstract}

\maketitle

\section{Introduction}

A fundamental problem in geometry is to determine the isometry group 
of a given manifold with a geometric structure. From a dynamical point
of view, an even more interesting problem is to determine for a given
Lie group $G$ the manifolds with geometric structures that admit an
action by isometries from $G$. Particularly interesting problems arise
when we assume $G$ to be a semisimple Lie group as it has been shown in
the work of Adams, Feres, Katok, Spatzier and Zimmer, among others.

For semisimple Lie groups, the existence of actions preserving geometric  
structures impose strong restrictions on the manifolds that admit such
actions. In a sense, this can be considered an extension of Margulis'
superrigidity theorem, since any action defines a representation of
the group into the diffeomorphism group of the manifold. However, the
techniques used to prove such restrictions are somehow more
complicated. 

When studying group actions it is very useful to distinguish those
satisfying suitable conditions. In this work we want to focus on
actions satisfying what is known as an engagement condition, with
particular emphasis on topologically and geometrically engaging
actions (see Definitions~\ref{def-top-eng} and
\ref{def-geom-eng1}). For any such restriction to be useful we need it 
to have two important features: 1) The condition must allow to obtain
interesting properties or apply known tools. 2) The condition must be
satisfied by most actions under study or it must be a consequence of
natural geometric/dynamic hypothesis. Theorems~\ref{thm-sz} and
\ref{thm-cq} are just two examples that show that topological and
geometric engagement satisfy the first feature, and plenty of other
results found in the references below provide more instances.

On the other hand, it turns out that topological engagement is
satisfied for actions that preserve a connection and a finite volume,
which is a pretty natural geometric condition. Even though this
statement is claimed to be true in several of the references below
there is no complete proof of this fact up to this date. The main goal
of this work is to provide such a proof to completely settle this
basic property of topologically engaging actions.

With respect to the organization of the article, in section 2 we
define the engaging conditions discussed above and describe some of
their applications. In section 3 we develop some of the basic lemmas
needed in our proof that come from a result of Gromov known as the
centralizer theorem. Section 4 contains the main result of this work,
Theorem~\ref{thm-main}, where we prove that for simple noncompact Lie
groups every analytic action preserving a unimodular rigid geometric
structure of algebraic type (among which we have the structures
consisting of a connection and a smooth measure) is topologically
engaging on an open conull dense set.  Finally, in section 5 we make
some observations that are brought to light from the ideas in the
proof of our main result.

The authors are grateful to Robert Zimmer for the enlightening remarks
that allowed to develop the ideas presented in this work.

\section{Actions and engagement conditions}

In the rest of this article the (Lie group) actions of structure
groups of fiber bundles are assumed to be on the right, except for the
actions of fundamental groups on universal covers which we assume to
be on the left. All other actions are assumed to be on the left as
well.

A standard technique used to study actions of Lie groups is to
require certain dynamical conditions to be fulfilled. For the case of
semisimple Lie groups of noncompact type it has been found that it is 
particularly useful to consider restrictions on the actions lifted to
coverings of the manifold being acted upon. In order to be more
precise we state without proof the following standard result:

\begin{prop}
Let $M$ be a connected manifold acted upon by a connected Lie group
$G$, then for any covering $\pi \colon M' \rightarrow M$ there is an
action of $\wt{G}$, the universal covering of $G$, that commutes with
the covering transformations and for which the covering maps are
equivariant, in other words we have $\pi(gm) = \pi_0(g)\pi(m)$ for all
$g\in \wt{G}$ and $m\in M'$, where $\pi_0 \colon \wt{G} \rightarrow G$
is the canonical covering map.
\end{prop}

If $G$ is a simple Lie group of noncompact type acting on a compact
manifold $M$, then the orbits are typically very complicated, unless
the action is trivial. However, it might be the case that after
lifting the action, as in the previous proposition, the orbits have
more manageable features. The following result (see \cite{rZ84})
provides a fundamental example to consider. We recall that a measure
preserving action is ergodic if the only measurable invariant sets are 
either null or conull (see \cite{rZ84} for more details).

\begin{prop}
\label{prop-moore}
Let $H$ be a simple noncompact Lie group and $\Gamma$ a lattice in $H$. 
If $G$ is noncompact closed subgroup of $H$, then the action of $G$ on
$H/\Gamma$ is ergodic with respect to any bi--invariant
measure on $H$.
\end{prop}

The action of $G$ on $H/\Gamma$ is very complicated for $G$
a proper subgroup as above, but such action can be lifted so that $G$
(not just its universal cover $\wt{G}$) acts on $H$ with closed orbits
that define a quotient $G \backslash H$ that has an analytic manifold
structure. We also remark that in the above result we can take $H$ to
be a semisimple noncompact Lie group as long as we require $\Gamma$ to
be an irreducible lattice (see \cite{rZ84}). Trying to capture this
special behavior, Zimmer has introduced in \cite{rZ89} the following
notion:

\begin{defn}
\label{def-top-eng}
A smooth action of a connected Lie group $G$ on a manifold $M$ is
called topologically engaging if there is some $\tilde{g}\in\wt{G}$
that acts tamely on $\wt{M}$ (i.e., has locally closed orbits) and
that projects to an element $g\in G$ that does not lie in a compact
subgroup. 
\end{defn}

\begin{rmk}
Observe that the action considered in the previous proposition is
topologically engaging. Also notice that for an action to be
topologically engaging it is enough for the action of $\wt{G}$ on
$\wt{M}$ to have locally closed orbits.
\end{rmk}

\begin{rmk}
For most applications it is enough for the above condition to hold on
suitable open subsets, so we will say that the action is topologically
engaging on a $G$--invariant open subset $U\subset M$ if there is a
$\tilde{g}\in\wt{G}$ whose orbits in $\wt{U}$ (the inverse image of
$U$ in $\wt{M}$) are locally closed and that projects to an element
$g\in G$ that does not lie in a compact subgroup. Notice that since
$U$ is open the locally closed condition does not depend on whether it
is being taken with respect to $\wt{U}$ or $\wt{M}$.
\end{rmk}

For a dynamical condition as topological engagement to be useful it
has to hold for an important family of actions and also it has to have 
interesting consequences. As remarked before, the main goal of this
work is to show that most ``geometric'' actions are topologically
engaging.
On the other hand, in \cite{SZ91} and \cite{rZ89} it has been
shown that topological engagement ensures rigid behavior for actions
of semisimple Lie groups of noncompact type. In particular, the
following result is the main step in the proof of Theorem A in
\cite{SZ91}:

\begin{theorem}
\label{thm-sz}
Let $G$ be a connected noncompact simple Lie group with finite center,
finite fundamental group and $\mathbb{R}$-$\rank(G)\geq 2$. Let $M$ be
a compact manifold and suppose that there is a topologically engaging
action of $G$ on $M$ preserving a finite measure. Then $\pi_1(M)$ is
not isomorphic to the fundamental group of any complete Riemannian
manifold $N$ with negative curvature bounded away from $0$ and
$-\infty$. 
\end{theorem}

In \cite{CQ99} we have introduced a dynamical condition for actions of 
semisimple Lie groups similar to Zimmer's topological engagement. Such 
condition comes from a particular way of measuring how the orbits in
the universal cover stretch out to infinity. The latter is more
precisely stated in the following:

\begin{defn}  
\label{def-geom-eng1}         
	Let $G$ be a connected semisimple Lie group of noncompact
	type, let $X$ be the symmetric space of noncompact type
	associated to $G$ and let $M$ be a compact Riemannian manifold  
	acted upon by $G$. Choose a Cartan decomposition   
	$\mathfrak{g} = \mathfrak{k} \oplus \mathfrak{m}$ for the Lie  
	algebra of $G$ (with $\mathfrak{k}$ a maximal compact subalgebra),
	and let $\mathfrak{m}_1$ be the unit ball
	in $\mathfrak{m}$ with respect to the Killing form of $\mathfrak{g}$.
	For $v \in \mathfrak{m}_1$ denote with  
	$g_t^v = \exp(tv)$ the one--parameter subgroup of $\widetilde{G}$  
	generated by $v$. The pointwise stretch of the action of $G$   
	on $M$ is the function defined by:  
	\begin{gather}  
	\pstre(G,M) \colon\mathfrak{m}_1\times \widetilde{M}
		\rightarrow \mathbb{R}  \notag \\   
	      (v,  x) \mapsto   
			\liminf_{t\to\infty}   
			\frac{d_{\widetilde{M}}(g_t^v x, x)}{t}     \notag  
	\end{gather}  
	where $\widetilde{G}$ acts on $\widetilde{M}$ by an arbitrary
	but fixed lift 
	of the action of $G$ on $M$. We say that the action of $G$ on $M$  
	has positive stretch if $\pstre(G,M)$ is a positive function.   
\end{defn}  

\begin{rmk}
Notice that the pointwise stretch of an action does not depend on the
choice of the lifted action to $\wt{M}$. Also notice that the pointwise 
stretch does depend on the choice of the Cartan decomposition of
$\mathfrak{g}$ and the Riemannian metric on $M$. However, it is easily 
seen that the condition of having positive stretch does not depend on
either of them. 
\end{rmk}

\begin{rmk}
In the definition of pointwise stretch, the distance
$d_{\widetilde{M}}(g_t^v x, x)$ measures the stretching of the orbit
of $x$ with respect to the one--parameter subgroup $g_t^v$ being
considered. On the other hand, from the basic theory of symmetric
spaces (see \cite{sH78}), if $x_0\in X$ is the point fixed by the
subgroup generated by $\mathfrak{k}$, then $t\mapsto g_t^v x_0$ is a
unit speed geodesic, and so we have $d_X(g_t^v x_0, x_0)
= t$. It follows that the limit that defines $\pstre(G,M)(v,x)$
compares as $t\rightarrow \infty$ the stretching out to infinity of
the orbits of $x$ with the stretching out to infinity of the geodesics 
in the space $X$. In particular, when the pointwise stretch is
positive, the orbits on $\wt{M}$ stretch out to infinity at least as
fast as they do in the symmetric space $X$.
\end{rmk}

\begin{rmk}
It is easily seen that an action as above has positive pointwise
stretch if and only if for every one--parameter subgroup $g_t$ of
$\wt{G}$ which does not map into a compact subgroup of $\Ad(G)$ we
have:
\begin{equation*}
	\liminf_{t\to\infty}   
	\frac{d_{\widetilde{M}}(g_t x, x)}{t} > 0
\end{equation*}
for every $x\in \wt{M}$. From this it is an easy matter to show that
an action with pointwise positive stretch is topologically engaging.
\end{rmk}

The notion of pointwise positive stretch is a natural translation to
actions of the notion of stretch considered in \cite{mG91} and
\cite{rQ97} for foliations. However, to obtain rigidity type results
for actions the following slightly stronger notion is needed:

\begin{defn} 
\label{def-geom-eng}        
	With the notation as in the previous definition,
	we say that the action of $G$ on $M$ is geometrically engaging 
	if for every sequence $(g_n)_n$ in $G$ such that $(g_nx_0)_n$ 
	is a quasi--ray in $X$ (the symmetric space associated to $G$) 
	for some (and hence any) $x_0 \in X$, the limit inferior:
	\begin{equation*}  
	\liminf_{n\to\infty} 
		\frac{d_{\widetilde{M}}(g_n x, x)}{d_X(g_n x_0, x_0)} > 0
	\end{equation*}  
	for every $x \in \widetilde{M}$.
\end{defn}

\begin{rmk}
Recall that a sequence $x_n$ in $X$ is called a quasi--ray if there
exist constants $A>0$ and $B\geq 0$ such that:
\begin{equation*}
	A^{-1}|m-n|-B \leq d_X(x_n,x_m) \leq  A|m-n|+B
\end{equation*}
for all $m,n\geq 0$.
\end{rmk}

\begin{rmk}
Observe that the condition of geometric engagement does not depend on
the choice of $x_0\in X$ or the Riemannian metric on $M$. However, the 
actual value of the limit inferior may depend on such choices.
\end{rmk}

\begin{rmk}
Since a geodesic in a symmetric space can always be seen as an orbit
by a one--parameter subgroup, it is easily shown that a geometrically
engaging action has positive pointwise stretch and so it is
topologically engaging as well.
\end{rmk}

Since geometric engagement is a stronger condition than topological
engagement it is expected to provide stronger rigidity type results,
which is the case as the following theorem shows (see
\cite{CQ99}):

\begin{theorem}
\label{thm-cq}
Let $G$ be a connected noncompact simple Lie group whose associated
symmetric space $X$ either has rank at least $2$ or is a
quaternionic or Cayley hyperbolic space and let $N$ be a compact
Riemannian manifold with nonpositive sectional curvature when $X$ has
rank $\geq 2$ and with nonpositive complexified sectional curvature
otherwise. Suppose $G$ has a geometrically engaging action on a
compact manifold $M$. If $\pi_1(M)\cong\pi_1(N)$, then there is an
isometric totally geodesic immersion $X \rightarrow N$. In particular, 
for $\rank(X)\geq 2$, the space $M$ cannot have the fundamental group
of a compact manifold with strictly negative sectional curvature.
\end{theorem}

\section{Gromov's centralizer theorem}

From a dynamical point of view, for an action to be more manageable it 
is desirable to have some sort of invariant structure on the manifold
being acted upon. As it is shown by Gromov in \cite{mG88}, for
superrigidity results, the most natural kind of structures that one can
consider are the rigid geometric structures. We will start this section by 
reviewing some of the basic definitions and notation used to describe these
geometric structures and we refer to \cite{mG88} and \cite{sK95} for further
details. 

Let $\Gl^{(k)}(n)$ be the Lie group of $k$--jets of diffeomorphisms of
$\R^n$ fixing the origin; this group is easily seen to be an algebraic
group.  Recall that a $k$--frame of an $n$-dimensional manifold $M$ at
a point $m\in M$ is the $k$--jet of a local diffeomorphism $(\R^n,0)
\rightarrow (M,m)$.
Then for such manifold $M$ there is a $\Gl^{(k)}(n)$--principal fiber
bundle $L^{(k)}(M)$, called the $k$--th order frame bundle, which
consists of the $k$--frames of $M$. In particular, for $k=1$ the group
$\Gl^{(1)}(n)$ is the usual general linear group and $L^{(1)}(M)$ is
the usual linear frame bundle of $M$. Let $Q$ be a smooth manifold that
admits a smooth action of $\Gl^{(k)}(n)$ and denote with $E^Q$ the
fiber bundle associated to $L^{(k)}(M)$ with standard fiber $Q$; then
a geometric structure on $M$ of order $k$ and type $Q$ is a smooth
section of $E^Q$, and such structure is called of algebraic type if
$Q$ is a real algebraic variety and the $\Gl^{(k)}(n)$--action is
algebraic as well. Every diffeomorphism $\phi$ of $M$ induces
corresponding bundle diffeomorphisms for both $L^{(k)}(M)$ and $E^Q$,
and for a geometric structure of order $k$ and type $Q$ we will say
that $\phi$ is an isometry or automorphism if the corresponding
section of $E^Q$ is equivariant with respect to such diffeomorphisms
induced by $\phi$. A smooth vector field on $M$ is called a Killing
vector field for a given geometric structure if its local flow acts by
local automorphisms for the structure.

A geometric structure $\omega$ of order $k$ and type $Q$ on $M$ is
called unimodular if for each $m \in M$ the $\Gl^{(k)}(n)$--orbit in
$Q$ of $\omega(m)$ has stabilizers whose images in $\Gl^{(1)}(n) =
\Gl(n)$ under the natural jet projection $\Gl^{(k)}(n) \rightarrow
\Gl^{(1)}(n)$ are contained in the group of matrices with determinant
$\pm 1$. It is easily seen that such structure defines a reduction of
the linear frame bundle $L^{(1)}(M)$ to the group of matrices with
determinant $\pm 1$, and so induce a volume density on $M$.

For a structure $\omega$ of order $k$ and type $Q$, $l\geq k$ and $x,y 
\in M$ we define $\Aut^{(l)}(\omega,x,y)$ to be the set of $l$--jets
of diffeomorphisms of $M$ taking $x$ to $y$ and $\omega(x)$ to
$\omega(y)$ up to order $l$; any such jet is called an infinitesimal
automorphism of $\omega$. We also denote $\Aut^{(l)}(\omega,m) =
\Aut^{(l)}(\omega,m,m)$, which is clearly a group.

Notice that whenever a manifold has a geometric structure there is a
corresponding geometric structure on its universal cover for which the
covering map is a local isometry and the fundamental group acts by
isometries as well. For such setup we will denote both geometric
structures with the same symbol.

\begin{defn}
A geometric structure $\omega$ is called $k$--rigid if for each $m\in
M$ and $l\geq k$, the natural jet projection map $\Aut^{(l)}(\omega,m) 
\rightarrow \Aut^{(k)}(\omega,m)$ is injective. A geometric structure
is called rigid if it is $k$--rigid for some $k$.
\end{defn}

\begin{rmk}
For $k$--rigid geometric structures the infinitesimal automorphisms at 
a point are completely determined by its $k$--jet at that point.
\end{rmk}

\begin{rmk}
It is easy to show that pseudoRiemannian metrics are rigid structures
of order $1$ and that affine connections are rigid structures of order
$2$. Also both are structures of algebraic type. Moreover, any finite
type structure in the sense of Cartan (see
\cite{sK95}) is rigid.
\end{rmk}

\begin{rmk}
Notice that all definitions and constructions above can be performed
replacing the smooth maps and manifolds by analytic ones, and that the 
corresponding remarks and properties mentioned above still hold true.
\end{rmk}

Gromov has extensively studied in \cite{mG88} the properties of rigid
structures that relate to the superrigid behavior of actions of
semisimple Lie groups. A very important result of such study is the
abundance of Killing vector fields for suitable rigid geometric
structures that commute with the action of a simple Lie group that
preserve any such structure. More precisely, we have the following
result that appears as Corollary 4.3 in \cite{rZ93}:

\begin{theorem}
\label{zimmer-centralizer}
Suppose $G$ is a noncompact simple Lie group with finite center acting
analytically and non trivially on a compact manifold $M$ preserving an 
analytic unimodular,
rigid, structure $\omega$ of algebraic type. Identify the Lie algebra
$\mathfrak{g}$ of $G$ with a Lie algebra of globally defined Killing
fields on $\wt{M}$ via the action of $M$. Let $\mathfrak{z}$ be the
centralizer of $\mathfrak{g}$ in the Lie algebra of globally defined
Killing vector fields on $\wt{M}$. For $x\in \wt{M}$, let
$\mathfrak{z}(x)$ and $\mathfrak{g}(x)$ be the images of
$\mathfrak{z}$ and $\mathfrak{g}$ respectively under the evaluation
map at $x$. Then for a.e.~$x\in \wt{M}$ we have
$\mathfrak{z}(x)\supset\mathfrak{g}(x)$. 
\end{theorem}

As an immediate corollary of this we have the following result. From
now on, given a manifold $M$ and a point $x\in M$ we will denote with
$\ev_x$ the evaluation map at $x$ for vector fields on $M$, and we
define $\wt\ev_x$ similarly for $\wt{M}$.

\begin{theorem}[Gromov's centralizer theorem]
Let $G$ be a simply connected, simple, noncompact Lie group with
finite center acting  on a compact manifold $M$ via
analytic diffeomorphisms, and preserving a unimodular, rigid, analytic
structure $\omega$ of algebraic type. Denote with $\GG$ the Lie
algebra of Killing vector fields on $\wt{M}$ induced by the
$G$--action.
 
Let $\VV$ denote the collection of all analytic vector fields $X\in
\X(\wt{M})$ such that
\begin{itemize}
	\item $X$ centralizes $\GG$, and
	\item $X$ is a Killing field for $\omega$.
\end{itemize}
Then:
\begin{itemize}
	\item[\upn{(}1\upn{)}] $\VV$ is $\pi_1(M)$-invariant. 
	\item[\upn{(}2\upn{)}] $\VV$ is finite dimensional.
	\item[\upn{(}3\upn{)}] $\VV$ centralizes $\GG$.
	\item[\upn{(}4\upn{)}] $T_x Gx \subset\wt{\ev}_x(\VV)$ 
			   for a.e.~$x\in\wt{M}$.
\end{itemize}
\end{theorem}

\begin{proof}
Conclusion $(1)$ follows from the fact that the actions of $\pi_1(M)$
and $G$ lifted to $\wt{M}$ commute, and also uses that $\omega$ in
$\wt{M}$ is $\pi_1(M)$--invariant. Conclusion $(2)$ follows from the
fact that a rigid structure has a finite dimensional Lie algebra of
Killing vector fields (see \cite{sK95} for a proof of this
claim in the case of finite type structures). Conclusion $(3)$ is immediate 
from the definition of $\VV$.

The only nontrivial claim is given by $(4)$, which follows from 
Theorem~\ref{zimmer-centralizer} since $\mathfrak{g}(x) = T_x Gx$.
\end{proof}

This remarkable result essentially states that, with the given
hypotheses, whenever the Lie algebra of Killing fields for the
geometric structure contains $\mathfrak{g}$ (the Lie algebra of $G$),
through the action of $G$, then there a whole new vector space of
Killing fields that in virtue of conclusion $(4)$ somehow still
contains $\mathfrak{g}$ and commutes with the original space of
Killing fields. A simple but important example to have in mind is the
action of $G$ as above on $G/\Gamma$, where $\Gamma$ is a
cocompact lattice. In this case we can take the geometric structure to
be the bi--invariant pseudoRiemannian metric on $G$ coming from the
Killing form and observe that the $G$--action lifts to the left
action of $G$ on itself so that $\GG$ is the Lie algebra of right
invariant vector fields and $\VV$ is the Lie algebra of left
invariant vector fields.

\begin{rmk}
Observe that the hypotheses of Gromov's centralizer theorem are
satisfied if we assume that the action leaves invariant an affine
connection and a volume density as long as we assume both to be
analytic. 
\end{rmk}

In the rest of the article we will assume that $M$ is a compact
manifold acted upon by a simple Lie group $G$ so that the hypotheses
of Gromov's centralizer theorem are satisfied. We will also denote with
$\Gamma$ the fundamental group of $M$.

\begin{rmk}
Since $\VV$ as above is $\Gamma$--invariant, Gromov's centralizer
theorem provides a representation $\rho \colon \Gamma \rightarrow
\Gl(\VV)$ that can be used to understand some of the properties of
$\Gamma$. As an example, in the following theorem due to Gromov (see
\cite{mG88} and \cite{rZ93}) the representation of $\Gamma$ is the one
provided by Gromov's centralizer theorem.
\end{rmk}

\begin{theorem}
\label{gromov-rep}
Let $M$ be a compact manifold acted upon by a simple Lie group $G$
satisfying the hypotheses of Gromov's centralizer theorem. Then there
is a representation $\rho \colon \Gamma \rightarrow \Gl(q)$ for some
$q$ such that the Zariski closure of $\rho(\Gamma)$ contains a group
locally isomorphic to $G$.
\end{theorem}

Such result imposes strong restrictions on the fundamental group of
$M$. For example, using a theorem of Moore (see \cite{rZ84}) it
follows that the fundamental group of $M$ cannot be amenable.
In the basic theory of Lie groups it is a well known fact that the
semisimple Lie groups of noncompact type provide a family of groups
which is completely disjoint from the family of amenable groups, and
every Lie group is built out of both families; more precisely, every
connected semisimple amenable Lie group is compact and every connected
Lie group is (up to a covering) the semidirect product of an amenable
group and a semisimple Lie group of noncompact type. Then the previous
result states that (under suitable restrictions) the fundamental group
of a manifold acted upon by a simple noncompact Lie group $G$ lies in
the same sort of family that contains $G$.

\section{Topological engagement}

In this section we will prove that actions satisfying the hypotheses
of Gromov's centralizer theorem are topologically engaging. Our main
tool will be the representation considered in
Theorem~\ref{gromov-rep}. In this section $M$ will denote a manifold
acted upon by $G$ satisfying the hypotheses of Gromov's centralizer
theorem, $\Gamma$ will denote the fundamental group of $M$ and $\rho$
the representation defined in Theorem~\ref{gromov-rep}.

Consider the product left action of $\Gamma$ on $\VV\times \wt{M}$
which defines a manifold $\Gamma\backslash(\VV\times\wt{M})$ that
fibers as a vector bundle over $M$. We will denote by $E^\VV$ the
total space of this vector bundle. It is easily seen that the
principal frame bundle $L(E^\VV)$ associated to $E^\VV$ is canonically
isomorphic to $\Gamma\backslash(\Gl(\VV)\times\wt{M})$. Notice that by
conclusion $(2)$ from Gromov's centralizer theorem, the space $\VV$ is
finite dimensional and so $\Gl(\VV)$ is a finite dimensional Lie
group. Also observe that the evaluation map defines an analytic vector
bundle map $\ev \colon E^\VV \rightarrow TM$. Since $G$ is simply
connected its left action on $M$ lifts to an action on the principal
bundle $\wt{M} \rightarrow M$ that commutes with the (left) action of
$\Gamma$. Being $E^\VV$ a fiber bundle associated to $\wt{M}
\rightarrow M$ there is an induced left action of $G$ on the bundles
$E^\VV \rightarrow M$ and $L(E^\VV) \rightarrow M$ by bundle
automorphisms, and it is easily checked that such actions are given by
$g[X,m] = [X,gm]$, where (for the case of $E^\VV$) $X\in \VV$, $m\in
\wt{M}$ and $g\in G$, with a similar expression for $L(E^\VV)$.

\begin{lemma}
The vector bundle map $\ev \colon E^\VV \rightarrow TM$ is
$G$--equivariant with respect to the actions induced on $E^\VV$ and
$TM$ as bundles associated to the principal bundle $\wt{M} \rightarrow
M$.
\end{lemma}

\begin{proof}
Given a diffeomorphism $\phi$ of $M$ and  a vector field $X$ on $M$
we clearly have $d\phi(X_m) = d\phi(X)_{\phi(m)}$ for every $m\in M$.
From this it follows that the evaluation map $\ev$ is
$G$--equivariant with respect to the natural action of $G$ on $TM$ and
the action of $G$ on $E^\VV$ given by $g[X,m] = [gX,gm]$, where $X\in
\VV$ and $gX$ denotes the vector field obtained by letting the
diffeomorphism defined by $g$ act on $X$. By conclusion $(3)$ of
Gromov's centralizer theorem, the action of $G$ on $\VV$ is trivial,
i.e.~$gX = X$ for all $X\in \VV$ and $g\in G$. Then $\ev$ is
$G$--equivariant for the action of $G$ on $E^\VV$ as associated bundle
of $\wt{M} \rightarrow M$.
\end{proof}

A very useful property of simple Lie groups is that their actions are
essentially locally free. More precisely we have the following result
that appears as Corollary 3.6 from \cite{rZ93} (see also
\cite{rZ86}). 

\begin{theorem}
Suppose $G$ is a simple noncompact Lie group acting non trivially on a
manifold $M$ with a finite invariant smooth measure. Then there is an
open dense conull $G$--invariant set for which the stabilizers are discrete.
\end{theorem}

\begin{rmk}
We recall that an action with discrete stabilizers is called locally
free and it is called essentially locally free if it is locally free
on a conull set.
\end{rmk}

A direct application of Frobenius theorem provides the following:

\begin{cor}
\label{cor-free}
Suppose $G$ is a simple noncompact Lie group acting non trivially on a
manifold $M$ with a finite invariant smooth measure. Then on an open
conull dense $G$--invariant subset $U_0$ of $M$ the $G$--orbits define
a smooth foliation and the tangent spaces to the orbits define a
smooth subbundle $T\OO$ of the tangent bundle $TU_0$. Moreover, if the
action is analytic, the foliation and $T\OO$ are analytic as well.
\end{cor}

This result allows us to improve the conclusions of Gromov's
centralizer theorem. 

\begin{lemma}
\label{gromov2}
Consider a group $G$ acting on a manifold $M$ as in Gromov's
centralizer theorem.  Then there is an open conull dense subset
$\wt{U}$ of $\wt{M}$, which is both $G$ and $\Gamma$--invariant, such
that conclusion $(4)$ from Gromov's centralizer theorem holds for
every $x\in \wt{U}$. Moreover, $G$ acts locally freely on $\wt{U}$.
\end{lemma}
\begin{proof}
Let $\wt{U}_0$ be an open $G$--invariant conull subset of 
$\wt{M}$ on which the $G$--action is locally free as provided by 
Corollary~\ref{cor-free}. Note that $\wt{U}_0$ may be assumed  
to be $\Gamma$--invariant. Consider the natural evaluation map $\wt{\ev} 
\colon \wt{U}_0 \times \VV \rightarrow T\wt{M}$. Since $\wt{\ev}$ is 
analytic it is easily seen that there is an open conull $\Gamma$ and 
$G$--invariant open subset $\wt{U}\subset \wt{U}_0$ (the complement of an
analytic set) on which $\rank(\wt{\ev}_x)$ is
maximal. From this it follows that $\bigcup_{x\in\wt{U}}\wt{\ev}_x(\VV)$
defines an analytic vector subbundle of $T\wt{M}$ on $\wt{U}$
and its continuity, together with that of the tangent bundle to the
orbits, implies that the set $A$ consisting of those points $x \in
\wt{U}$ such that $T_xGx$ is not contained in $\wt{\ev}_x(\VV)$ is open
in $\wt{U}$. But Gromov's centralizer theorem implies that $A$ is null in
$\wt{U}_0$ and since the measure on $\wt{M}$ is smooth it must
be that $A$ is empty. Hence, $\wt{U}$ satisfies conclusion $(4)$ and
has the required properties.
\end{proof}

The following will also prove to be a useful property to consider.

\begin{lemma}
\label{lem-trivial}
Let $G$ and $M$ be as in Corollary~\ref{cor-free}. Then on an open
conull dense $G$--invariant subset $U_0$ of $M$ both $T\OO$ and its
frame bundle $L(T\OO)$ are trivial. Moreover, there is a
trivialization $L(T\OO) \cong \Gl(\mathfrak{g}) \times U_0$, where
$\mathfrak{g}$ denotes the Lie algebra of $G$, so that the $G$--action 
is given by $g(A,m) = (\Ad_G(g)\circ A, gm)$ for every $A\in
\Gl(\mathfrak{g})$, $g\in G$ and $m \in U_0$.
\end{lemma}

\begin{proof}
Let $U_0$ be an open subset of $M$ as in Corollary~\ref{cor-free} and
for every $X\in\mathfrak{g}$ let $X^*$ be the vector field on $U_0$
induced by $X$, i.e.~$X^*$ is given at a point $m\in U_0$ by
\begin{equation*}
	X^*_m = \ddt\Bigr|_{t=0} (\exp(tX)m)
\end{equation*}
Consider the bundle map defined by
\begin{align*}
	\alpha \colon \mathfrak{g}\times U_0 &\rightarrow T\OO \\ 
		(X,m) &\mapsto X^*_m
\end{align*}
Since the action is locally free every basis $(X_i)_i$ of $\g$ induces at every
point $m \in U_0$ a family of vectors $((X_i)^*_m)_i$ that also
defines a base for the tangent space to the orbit at $M$, i.e.~to the
fiber of $T\OO$ at $m$. In particular, $\alpha$ trivializes $T\OO$ and 
$L(T\OO)$ is trivial as well.

On the other hand, for $X\in \mathfrak{g}$, $m\in U_0$ and $g\in G$ we
have: 
\begin{align*}
	g X^*_m &= g\ddt\Bigr|_{t=0}(\exp(tX)m)	\\
		&= \ddt\Bigr|_{t=0}(g\exp(tX)m)	\\
		&= \ddt\Bigr|_{t=0}(g\exp(tX)g^{-1} gm) \\
		&= \ddt\Bigr|_{t=0}(\exp(t\Ad_G(g)(X))gm) \\
		&= \Ad_G(g)(X)^*_{gm}
\end{align*}
In particular, the trivialization $\alpha$ is $G$-equivariant for the
action on $\mathfrak{g}\times U_0$ given by $g(X,m) =
(\Ad_G(g)(X), gm)$.

Now observe that if $\alpha_m$ denotes the linear map at the fibers
over $m$ induced by $\alpha$, then the trivialization of $L(T\OO)$ is
given by:
\begin{align*}
	\beta \colon \Gl(\mathfrak{g})\times U_0 &\rightarrow L(T\OO) \\ 
		(A,m) &\mapsto \alpha_m\circ A
\end{align*}
where we have taken $\mathfrak{g}$ as the standard fiber of $T\OO$
when considering $L(T\OO)$ as a frame bundle for $T\OO$. We then have
that for $A\in\Gl(\mathfrak{g})$, $m\in U_0$ and $g\in G$:
\begin{align*}
	(g(\alpha_m\circ A))(v) &= g(\alpha_m\circ A)(v)	\\
				&= gA(v)^*_m			\\
				&= (\Ad_G(g)(A(v)))^*_{gm}	\\
				&= (\alpha_{gm}\circ\Ad_G(g)\circ A)(v)
\end{align*}
for all $v\in \mathfrak{g}$. It follows that $\beta$ is
$G$--equivariant for the $G$--action on $\Gl(\mathfrak{g})\times U_0$
given by  $g(A,m) = (\Ad_G(g)\circ A, gm)$ for every $A\in
\Gl(\mathfrak{g})$, $g\in G$ and $m \in U_0$.
\end{proof}

For every $x \in M$, denote with $\ev_x$ the linear map given by the
bundle map $\ev \colon E^\VV \rightarrow TM$ at the fibers over $x$,
and define the subspaces of $E^\VV_x$ (the fiber of $E^\VV$ at $x$)
given by $T_x = \ev_x^{-1}(T_x Gx)$ and $K_x = \ev_x^{-1}(0)$. The
following result states that over a suitable open set the spaces $T_x$ 
and $K_x$ define analytic vector bundles. 

\begin{prop}
\label{prop-bundles}
Let $G$ be a Lie group acting on a manifold $M$ satisfying the
hypotheses of Gromov's centralizer theorem. Then there is a
conull dense open $G$--invariant subset $U$ of $M$ so that the
following subsets of $E^\VV$ define analytic subbundles of $E^\VV$
over $U$:
\begin{align*}
	T|_U &= \bigcup_{x\in U} T_x \\
	K|_U &= \bigcup_{x\in U} K_x
\end{align*}
\end{prop}

\begin{proof}
Let $U$ be an open subset of $M$ whose inverse image under the natural
projection $\wt{M}\rightarrow M$ satisfies the conclusion of
Lemma~\ref{gromov2}.  Then we have for every $x \in U$ that
$\ev_x(E^\VV_x) \supset T_x\OO$, where the latter denotes the fiber of
$T\OO$ at $x$. Since $\ev$ is an analytic bundle map and since
\begin{align*}
	\dim T_x &= \dim\VV + \dim G - \rank(\ev_x) \\
	\dim K_x &= \dim\VV - \rank(\ev_x)
\end{align*}
it is enough to observe that, by the proof Lemma~\ref{gromov2}, the map $x
\mapsto \rank(\ev_x)$ is constant on $U$ to obtain the conclusion.
\end{proof}

Finally, we prove the main result of this article:

\begin{theorem}
\label{thm-main}
Let $M$ be a compact analytic manifold acted upon on the left by a
simply connected simple non--compact Lie group $G$ with finite center
preserving a rigid unimodular analytic geometric structure of
algebraic type \upn{(}e.g.~a connection and a volume form, both
analytic\upn{)}. Then there is an open conull dense $G$--invariant open subset
$U$ of $M$ whose lift to $\wt{M}$ has locally closed $G$--orbits. In
particular, the $G$--action is topologically engaging on $U$.
\end{theorem}
\begin{proof}
By Corollary~\ref{cor-free} and Proposition~\ref{prop-bundles} there
is an open conull dense $G$--invariant subset $U$ of $M$ so that the
action on $U$ is locally free and the sets $T|_U$ and $K|_U$ define
analytic subbundles of $E^\VV$ over $U$. From the proof of the previous
results it follows easily that $U$ is $\Gamma$--invariant. 

Let us denote with $N$, $m$ and $n$ the ranks of the bundles
$E^\VV$, $T|_U$ and $K|_U$, respectively. Hence, it is clear that the
subset of $L(E^\VV)$ given by $L(T,K)|_U = \{ u \in L(E^\VV)|_U \;\vert\;
u(\R^m) = T_{p(u)},\ u(\R^n) = K_{p(u)}\}$ (each $u \in L(E^\VV)$
is considered as a frame of $E^\VV$) is an analytic principal
subbundle of $L(E^\VV)$ over $U$, where $L(E^\VV)|_U$ is the
restriction of $L(E^\VV)$ to $U$ and $p \colon L(E^\VV) \rightarrow
M$ is the canonical projection. Also observe that the bundles $T|_U$,
$K|_U$ and $L(T,K)|_U$ are all $G$--invariant.

Let $\wt{U}= \pi^{-1}(U)$ be the open subset of $\wt{M}$ where $\pi$
is the projection of $\wt{M}$ onto $M$. Given $m_0 \in \wt{U}$ we will 
prove that the $G$--orbit of $m_0$ is closed in $\wt{U}$.

For $A_0 \in \Gl(\VV)$ consider the map $\lambda \colon \wt{M} 
\rightarrow L(E^\VV) = \Gamma\backslash(\Gl(\VV)\times \wt{M})$ given
by $\lambda(m) = [A_0,m]$. Observe that as a subset of $L(E^\VV) =
\Gamma\backslash(\Gl(\VV)\times \wt{M})$ we can identify $L(T,K)|_U =
\{ [A,m] \;\vert\; m\in \wt{U}, A\in \Gl(\VV), A(\TT_0) = \TT_m, A(\KK_0) =
\KK_m \}$, where $\TT_m = \{ X\in \VV \;\vert\; X_m\in T\OO \}$, $\KK_m = \{
X\in \VV \;\vert\; X_m = 0 \}$ and $\TT_0$, $\KK_0$ are the corresponding
spaces at a fixed base point of $\wt{U}$. 

Since $G$ acts trivially on
$\VV$ it is straightforward to check that $\TT_{gm} = g\TT_m = \TT_m$
and $\KK_{gm} = g\KK_m = \KK_m$ for every $m\in \wt{U}$ and $g \in
G$. Since the fibers of $L(E^\VV)$ are acted upon by the structure
group transitively, we can choose $A_0$ so that $\lambda(m_0)\in
L(T,K)|_U$, and we then have in particular that 
\begin{equation}
\label{eq-m0}
\begin{aligned}
	A_0(\TT_0) &= \TT_{m_0}	\\
	A_0(\KK_0) &= \KK_{m_0}
\end{aligned}
\end{equation}
But since $\lambda$ is a $G$--equivariant map and $L(T,K)|_U$ is
$G$--invariant we conclude that $\lambda(Gm_0)\subset L(T,K)|_U$.

We claim that we further have $\lambda(cl_{\wt{U}}(Gm_0))
\subset L(T,K)|_U$ where $cl_{\wt{U}}(Gm_0)$ is the closure of
$Gm_0$ in $\wt{U}$. To see this we need to show that for every
sequence $(g_n m_0)_n$ in $\wt{U}$ that converges to $m_1\in\wt{U}$ we
have $\lambda(m_1) \in L(T,K)|_U$. In other words, we need to show
that $A_0(\TT_0) = \TT_{m_1}$ and $A_0(\KK_0) = \KK_{m_1}$, and
equation~(\ref{eq-m0}) makes this equivalent to showing that
$\TT_{m_0} = \TT_{m_1}$ and $\KK_{m_0} = \KK_{m_1}$.
Choose $X\in \TT_{m_0}$ and observe that $X\in T_{g_n m_0}$ for every
$n$. In other words, $X_{g_n m_0} \in T\OO$ and since $g_n m_0
\rightarrow m_1$ in $\wt{U}$ and $T\OO$ is the tangent bundle to a
foliation in $\wt{U}$ we conclude that $X_{m_1} \in T\OO$, i.e.~ $X\in
\TT_{m_1}$. Since both spaces have the same dimension it follows that
$\TT_{m_1} = \TT_{m_0}$, and a similar argument proves the claim for
$\KK$. 

Observe that on $U$ the bundle map $\ev$ induces an isomorphism between
the quotient bundle $T|_U/K|_U$ and $T\OO$ ($T\OO$ is defined on $U$
only) which is essentially a consequence of the definitions of $T|_U$
and $K|_U$. If we denote with $L(T\OO)$ the principal fiber bundle
over $U$ associated to $T\OO$, then it is easy to check that the map:
\begin{align*}
	\mu \colon L(T,K)|_U &\to L(T\OO)	\\
		u &\mapsto \tilde{u}
\end{align*}
where $\tilde{u}$ denotes the isomorphism $\R^{m-n} \to T_{p(u)}p(u)
G$ induced by $u$ and $\ev_{p(u)}$ ($p(u)\in M$ is the base point
of $u$), defines a homomorphism of principal bundles which is easily
seen to be $G$--equivariant.

By Lemma~\ref{lem-trivial}, the bundle $L(T\OO)$ over $U$ is
$G$--equivariantly analytically equivalent to $\Gl(\g) \times U$
where the $G$--action on the latter is given by $g(A,m) =
(\Ad_G(g)\circ A, gm)$.

Now let $m_1 \in cl_{\wt{U}}(m_0 G)$, so there is a sequence $(g_n
m_0)_n$ that converges to $m_1$. From the above it follows that
$g_n\mu\circ\lambda(m_0) \rightarrow \mu\circ\lambda(m_1)$, i.e.~we
have $g_n(A_1,m_0) \to (A_2,m_1)$ in $L(T\OO)$ for some $A_1, A_2 \in
\Gl(\mathfrak{g})$, so that given the above action on $L(T\OO) \cong
\Gl(\g) \times U$ it follows that $(\Ad_G(g_n))_n$ converges in
$\Ad_G(G)$ and since $G$ has finite center we can replace $(g_n)_n$ by
a subsequence to assume that $(g_n)_n$ converges to some $g\in
G$. From this it follows that $m _1 = gm_0$ and so the orbit $Gm_0$
is closed in $\wt{U}$.
\end{proof}

A straightforward consequence is given by the following result.

\begin{cor}
Let $G$ be a group acting on a manifold $M$ as in Theorem~\ref{thm-main}.
If the $G$--action on $M$ is minimal, i.e.~all orbits are dense, 
then the $G$--action on $M$ is topologically engaging.
\end{cor}

\section{Further developments}

As it is observed in the previous sections, geometric engagement is a
condition stronger to but closely related to topological engagement. A 
natural problem is to determine whether or not connection preserving
actions as those studied here are geometrically
engaging. In~\cite{CQ99} it has been proved that essentially all known 
actions of that sort are geometrically engaging, but the problem still 
remains open.

Nevertheless, its the authors belief that connection preserving
actions are geometrically engaging. Moreover, we expect that some
(nontrivial) extensions of the arguments in this work might allow to
prove this fact. Notice that topological engagement is a purely
topological condition while geometric engagement requires the choice
of a Riemannian metric and some distance estimates, so the proof of
geometric engagement should be considerably more complicated. On the
other hand, the applications that would arise from such fact would be
stronger than some of those obtained from topological engagement, as
it has been remarked above and in \cite{CQ99}.

\end{document}